\documentclass{amsart}
\numberwithin{equation}{section}

\usepackage{bbm}
\usepackage{comment}
\usepackage{verbatim}
\usepackage{mathrsfs}
\usepackage{mathtools}
\usepackage{latexsym}
\usepackage{epsfig}
\usepackage{amsmath}
\usepackage[usenames,dvipsnames]{xcolor}
\usepackage{graphics}
\usepackage[utf8]{inputenc}
\usepackage{amssymb}
\usepackage{amsthm}
\usepackage[alphabetic]{amsrefs}
\usepackage{amsopn}
\usepackage{amscd}
\usepackage[all,knot]{xy}
\xyoption{all}
\usepackage{rotating}
\usepackage{tikz}
\usetikzlibrary{calc}
\usepgflibrary{shapes.geometric}
\usepgflibrary{shapes.misc}
\usetikzlibrary{positioning}
\usetikzlibrary{decorations}
\usetikzlibrary{decorations.pathreplacing}
\usepackage{hyperref}
\usepackage{tikz}
\usepackage{tikz-cd}

\theoremstyle{plain}
\newtheorem{thm}{Theorem}[section]
\newtheorem{lem}[thm]{Lemma}
\newtheorem{prop}[thm]{Proposition}
\newtheorem{cor}[thm]{Corollary}

\newtheorem*{thm*}{Theorem}
\newtheorem*{lem*}{Lemma}
\newtheorem*{prop*}{Proposition}
\newtheorem*{cor*}{Corollary}

\theoremstyle{definition}
\newtheorem{defn}[thm]{Definition}
\newtheorem*{defn*}{Definition}

\newtheorem{ex}[thm]{Example}
{}
\newtheorem{rem}[thm]{Remark}
\newtheorem*{rem*}{Remark}

{}
{}
{}
{}
\newtheorem{qn}[thm]{Question}{}

\theoremstyle{remark}
{}
{}
{}

\def\to{\longrightarrow} 

\def\AA{\mathbb{A}}

\def\CC{\mathbb{C}}

\def\EE{\mathbb{E}}

\def\ZZ{\mathbb{Z}}
\DeclareUnicodeCharacter{221E}{$\infty$}

\def\sfC{\mathsf{C}}
\def\sfD{\mathsf{D}}

\def\sfH{\mathsf{H}}

\def\sfK{\mathsf{K}}

\def\sfT{\mathsf{T}}

\def\mcE{\mathcal{E}}

\def\sg{\mathrm{sg}}

\DeclareMathOperator{\Hom}{Hom}

\DeclareMathOperator{\modu}{\mathsf{mod}}

\DeclareMathOperator{\gr}{\mathsf{gr}}

\DeclareMathOperator{\Perf}{\mathsf{Perf}}

\DeclareMathOperator{\RHom}{\mathbf{R}Hom}

\DeclareMathOperator{\cone}{cone}

\DeclareMathOperator{\thick}{thick}

\DeclareMathOperator{\rad}{rad}

\def\dgcatk{\mathsf{dgcat}_k}

\definecolor{internationalkleinblue}{rgb}{0.0, 0.18, 0.65}

\title{Homotopy Invariants for gradable finite dimensional algebras}

\author{Sira Gratz}
\thanks{The first author was supported by a research grant (VIL42076) from VILLUM FONDEN and by EPSRC (Grant Number EP/V038672/1). All the authors are grateful to the Hausdorff Research Institute for Mathematics where this work began during the Junior Trimester Program: New Trends in Representation Theory.}
\address{Sira Gratz, Aarhus University, Department of Mathematics, Ny Munkegade 118, bldg. 1530
DK-8000 Aarhus C, Denmark
}
\email{sira@math.au.dk}
\urladdr{https://sites.google.com/view/siragratz}

\author{Theo Raedschelders}
\address{Theo Raedschelders, The Island, Natiestraat 2, B-2060 Antwerpen, Belgium}

\author{\v{S}pela \v{S}penko}
\address{\v{S}pela \v{S}penko, D\'epartement de Math\'ematique, Universit\'e Libre de Bruxelles, Campus de la Plaine CP 213, Bld du Triomphe, B-1050 Bruxelles, Belgium}
\email{spela.spenko@ulb.be}

\author{Greg Stevenson}
\address{Greg Stevenson, Aarhus University, Department of Mathematics, Ny Munkegade 118, bldg. 1530
DK-8000 Aarhus C, Denmark
}
\email{greg@math.au.dk}
\urladdr{https://sites.google.com/view/gregstevenson}


\begin{document}

\begin{abstract}
We show that for a gradable finite dimensional algebra the perfect complexes and bounded derived category cannot be distinguished by homotopy invariants.
\end{abstract}

\maketitle

\setcounter{tocdepth}{1}
\tableofcontents



\section{Introduction}

Given a finite dimensional algebra $\Lambda$ one can associate to it the (pre)triangulated (dg) categories $\Perf \Lambda$ and $\sfD^{\mathrm{b}}(\modu \Lambda)$. The former is the completion of $\Lambda$ under finite homotopy colimits and remains finite dimensional in an appropriate sense. The bounded derived category is, in principle, larger and coincides with $\Perf \Lambda$ precisely when $\Lambda$ has finite global dimension. In fact, $\sfD^{\mathrm{b}}(\modu \Lambda)$ is dual to $\Perf \Lambda$ and although it may not be finite dimensional it has the complementary finiteness property of being smooth i.e.\ of finite global dimension.

Let us assume that $\Lambda$ has infinite global dimension. A natural way to study these distinct categories is through their invariants, such as $K$-theory, Hochschild (co)homology, and so on. On one hand, it would be extremely interesting if differences between the invariants of $\Perf \Lambda$ and $\sfD^{\mathrm{b}}(\modu \Lambda)$ shed light on how singular $\Lambda$ was, and on the other hand we might expect some relationship between their invariants coming from the duality relating these categories.

We know exactly what to expect from $\Perf \Lambda$. Invariants of dg categories are insensitive to the process of adding finite homotopy colimits and so $\Perf \Lambda$ is indistinguishable from $\Lambda$. The bounded derived category is more mysterious, but we know how to compute some invariants for $\sfD^{\mathrm{b}}(\modu \Lambda)$. For instance, it is well known that $K_0$ of both categories is free of the same finite rank, and by a result of Lowen and Van den Bergh \cite{LVdBHH}*{Theorem~4.4.1} the Hochschild cohomology of $\sfD^{\mathrm{b}}(\modu \Lambda)$ and $\Perf \Lambda$ agree (cf.\ \cite{GoodbodyHH} for a recent persepctive in terms of duality). We also know that some invariants, for example Hochschild homology, distinguish them.

In this article we prove that the $\AA^1$-homotopy invariants of $\Perf \Lambda$ and $\sfD^{\mathrm{b}}(\modu \Lambda)$ coincide provided $\Lambda$ admits a non-negative grading such that $\Lambda_{\geq 1}$ is the Jacobson radical. The identification is not induced by the natural inclusion, but rather by (more or less) showing both categories are $\AA^1$-motivically just $\Lambda/\rad(\Lambda)$. As an application we are able to compute $\AA^1$-homotopy invariants of the singularity category of such a $\Lambda$, generalizing the results of \cite{MR4103346} for self-injective algebras. We conclude the article by discussing what we know in the absence of such a grading.



\section{Conventions and preliminaries}

Throughout, we work over a fixed base field $k$ which we assume is algebraically closed. We work with right modules and right dg modules. We briefly recall the main definitions we will need. Our aim is simply to situate the reader and fix notation.


\subsection{Graded algebras}

We fix $\Lambda$ to be a finite dimensional $k$-algebra admitting a semi-simple grading, that is, a non-negative $\ZZ$-grading with semi-simple degree zero part. So, we have
\[
\Lambda = \bigoplus_{i=0}^n \Lambda_i \quad \text{ and } \quad \Lambda_0 = \Lambda/\rad(\Lambda) = S
\]
where $\rad(\Lambda)$ denotes the Jacobson radical of the ungraded algebra $\Lambda$ (we see immediately the radical is gradable and agrees with the graded radical of the graded algebra) and $n$ is some natural number.

We denote by $\gr \Lambda$ the category of finite dimensional graded $\Lambda$-modules and degree $0$ morphisms. This category has a natural action of $\ZZ$ by automorphisms: for $i\in \ZZ$ and a graded module $N$ we define $N(i)$ to be the graded module with $N(i)_j = N_{i+j}$, i.e.\ we just reindex the grading and the action on morphisms is also given by reindexing.



\subsection{Invariants}

Throughout we work with dg categories over $k$. For a finite dimensional algebra $\Lambda$ we denote by $\Perf \Lambda$ the dg category of bounded complexes of finitely generated projective modules and by $\sfD^{\mathrm{b}}(\modu \Lambda)$ the dg category of complexes of finitely generated projectives with finite dimensional total cohomology. These are enhancements of the derived category of perfect complexes and bounded derived category respectively.

We now recall the definition of an $\AA^1$-homotopy invariant. Denote by $\dgcatk$ the category of (essentially) small dg categories over $k$, i.e.\ this is the category with objects the small dg categories and morphisms given by isomorphism classes of dg functors. In addition we fix some triangulated category $\sfT$.

A localization sequence of dg categories is the inclusion of a thick subcategory followed by the corresponding Verdier quotient (up to Morita equivalence). Some further details and equivalent formulations can be found in \cite{KellerDG}*{Theorem~4.11}.

\begin{defn}
A functor $\EE\colon \dgcatk \to \sfT$ is a \emph{localizing invariant} if
\begin{itemize}
\item[(1)] $\EE$ sends derived Morita equivalences to isomorphisms, in particular for any dg category $\sfC$ the canonical inclusion $\sfC\to \Perf(\sfC)$ is sent to an isomorphism by $\EE$;
\item[(2)] $\EE$ sends localization sequences of dg categories to triangles.
\end{itemize}
A localizing invariant $\EE$ is $\AA^1$\emph{-homotopy invariant} if moreover
\begin{itemize}
\item[(3)] $\EE$ inverts the canonical inclusion
\begin{displaymath}
\sfC \to \sfC[t] = \sfC\otimes_k k[t]
\end{displaymath}
for every dg category $\sfC$, where $k[t]$ is concentrated in degree $0$.
\end{itemize}
\end{defn}

\begin{rem}
It would be very natural to ask that $\EE$ actually takes values in a presentable stable $\infty$-category and moreover preserves filtered homotopy colimits. This occurs in examples, but we will not require this strengthening for our arguments.
\end{rem}







\section{Homotopy invariants of bounded and perfect complexes}

In this section we let $\Lambda$ denote a basic finite dimensional algebra over an algebraically closed field $k$. Moreover, we assume that $\Lambda$ admits a semi-simple grading. Let $S = \Lambda_0 = \Lambda/ \rad(\Lambda)$ denote the top of $\Lambda$, with $n = \dim \Lambda_0$ simple summands.

We begin with a preparatory lemma.

\begin{lem}
Let $M$ and $N$ be gradable finite dimensional $\Lambda$-modules. Fix gradings on $M$ and $N$ such that $M$ is concentrated in non-negative degrees and $N$ is concentrated in non-positive degrees. Then
\[
\RHom_\Lambda(M,N) = \bigoplus_{j\leq 0}\RHom_{\gr \Lambda}(M,N(j)).
\]
\end{lem}
\begin{proof}
We observe that $\Hom_{\gr \Lambda}(\Lambda, N(j)) \cong N_j$ and so $\oplus_{j\in \ZZ}\Hom_{\gr \Lambda}(\Lambda, N(j)) \cong N$. Hence, using a presentation for $M$, it follows that $\Hom_\Lambda(M,N) = \oplus_{j\in \ZZ}\Hom_{\gr \Lambda}(M, N(j))$. Taking right derived functors yields the desired formula, except for the degree bound on the sum.

Let $P^\bullet$ be a minimal graded projective resolution of $M$. If we forget the grading this gives a minimal projective resolution of $M$ sans grading. As $M$ is generated in non-negative degrees and $\Lambda$ is non-negatively graded, we have that each $P^i$ is also concentrated in non-negative degrees (cf.\ \cite{BurkeStevenson}*{Lemma~3.10}). Using this resolution to compute $\RHom_{\gr \Lambda}(M,N(j))$ we see that it is acyclic unless $j\leq 0$.
\end{proof}

With this in hand we proceed to the main theorem.

\begin{thm}\label{T:EDb=EPerf}
	For any $\AA^1$-homotopy invariant $\EE$ we have
	\[
		\EE(\sfD^{\mathrm{b}}(\modu \Lambda)) \cong \EE(k)^{\oplus n} \cong \EE(\Perf(\Lambda)).
	\]
\end{thm}

\begin{proof}
	The semi-simple grading on $\Lambda$ induces a decomposition of $\mcE = \RHom_\Lambda(S,S)$ as
	\[
		\mcE \cong \bigoplus_{j \in \ZZ} \RHom_{\gr \Lambda}(S,S(j)) \cong \bigoplus_{j \leq 0} \RHom_{\gr \Lambda}(S,S(j)),
	\]
	by the previous lemma. We note that this decomposition is compatible with the dg algebra structure on $\mcE$: the first map is a dg algebra quasi-isomorphism and the rightmost term is a subalgebra of the middle one. The degree $0$ part of $\mcE$ with respect to this internal grading is 
	\[
	 \mcE_0 \cong \RHom_{\gr \Lambda} (S,S).
	 \]
Let us now consider $\mcE_0$ with its natural cohomological grading. It is a differential graded algebra with cohomology concentrated in degree $0$ and hence formal i.e.\ quasi-isomorphic to its degree $0$ cohomology $\Hom_{\gr \Lambda} (S,S)$. 
	Using \cite{TVdB}*{Lemma~6.6} for the first isomorphism and the fact that $\EE$ preserves quasi-isomorphisms for the second and third, we obtain
	\[
		\EE(\mcE) \cong \EE(\mcE_0)  \cong \EE(\RHom_{\gr \Lambda}(S,S)) \cong \EE(\Hom_{\gr \Lambda}(S,S)) \cong \EE(k)^{\oplus n}.
	\]
	On the other hand, employing the fact that $\EE$ preserves derived Morita equivalences for the first isomorphism, and \cite{TVdB}*{Lemma~6.6} for the second, yields
	\[
		\EE(\Perf \Lambda) \cong \EE(\Lambda) \cong \EE(\Lambda_0) \cong \EE(k)^{\oplus n}.
	\]
We are now essentially done: since every finite dimensional $\Lambda$-module has a finite composition series we have $\sfD^{\mathrm{b}}(\modu \Lambda) \cong \thick(S) \cong \Perf \mcE$. Given the above computations we obtain
\[
	\EE(\sfD^\mathrm{b}(\modu \Lambda)) \cong \EE(\Perf( \mcE)) \cong\EE( \mcE) \cong \EE(k)^{\oplus n}
\]
which completes the proof.
\end{proof}

A number of comments are in order. First, let us remark that the isomorphism of the theorem is not, in general, induced by the inclusion $\iota\colon \Perf \Lambda \to \sfD^\mathrm{b}(\modu \Lambda)$. The invariants of the singularity category
\[
\sfD_{\sg}(\Lambda) = \sfD^\mathrm{b}(\modu \Lambda) / \Perf \Lambda
\]
measure exactly this phenomenon.

Denote by $C_\Lambda$ the Cartan matrix of $\Lambda$, i.e.\ the matrix encoding the multiplicity of the simple modules in the indecomposable projectives. For us the relevant definition is that $C_\Lambda$ is the matrix describing the morphism
\[
\begin{tikzcd}
\ZZ^n \arrow[r, "\sim"] & \sfK_0(\Perf \Lambda) \arrow[r, "\sfK_0(\iota)"] & \sfK_0(\sfD^\mathrm{b}(\modu \Lambda)) \arrow[r, "\sim"] & \ZZ^n
\end{tikzcd}
\]
where the identifications with $\ZZ^n$ are given by taking the basis consisting of the classes of the indecomposable projective and simple modules respectively.

\begin{cor}
	Let $\EE$ be an $\AA^1$-homotopy invariant. We have
	\[
		\EE\left(\sfD_{\sg}(\Lambda)\right) \cong \cone \left(\xymatrix{\EE(k)^{\oplus n} \ar[r]^-{C_\Lambda} & \EE(k)^{\oplus n}}\right)
	\]
	
\end{cor}

\begin{proof}
By Theorem \ref{T:EDb=EPerf} applying $\EE$ to the localization sequence 
\[
	\xymatrix{\Perf \Lambda \ar[r]^-\iota & \sfD^{\mathrm{b}} (\Lambda)\ar[r] & \sfD_{\sg}(\Lambda)}
\] 
yields the triangle
\[
	\xymatrix{\EE(k)^{\oplus n} \ar[r]^-{\EE(\iota)} & \EE(k)^{\oplus n}\ar[r] & \EE(\sfD_{\sg}(\Lambda))}
\] 
and so the statement comes down to identifying $\EE(\iota)$. This can be done by applying \cite{MR3342391}*{Proposition~2.8}, which reduces the identification to computing $\sfK\sfH_0(\iota)$ where $\sfK\sfH$ is homotopy $K$-theory. For both the perfect and bounded complexes we have agreement of $\sfK\sfH_0$ and $\sfK_0$ and so this map is the Cartan matrix essentially by definition. 

\end{proof}

\begin{rem}
This gives a generalization, with a new proof, of the computation \cite{MR4103346}*{Theorem~3.4.2} for the self-injective case.
\end{rem}

\begin{ex}
We obtain new examples of dg categories whose $\AA^1$-homotopy invariants are uniformly trivial. For instance, the algebras described in \cite{MR801315} Examples~8 and 9 have infinite global dimension but Cartan matrices of determinant $1$ and $-1$ respectively. It follows that their singularity categories are non-trivial but have trivial $\AA^1$-motives.
\end{ex}

It is natural to ask if the conclusion of Theorem~\ref{T:EDb=EPerf} is valid for \emph{any} finite dimensional algebra.

\begin{qn}
Let $\EE$ be an $\AA^1$-homotopy invariant and $\Gamma$ a finite dimensional algebra. Do we have
\[
\EE(\sfD^{\mathrm{b}}(\modu \Gamma)) \cong \EE(k)^{\oplus n} \cong \EE(\Perf(\Gamma))?
\]
If so, is there a structural proof of this fact that illuminates why?
\end{qn}

This question has content: it is not the case that any finite dimensional algebra satisfies the hypotheses of Theorem~\ref{T:EDb=EPerf}. For instance, \cite{ungradable} gives examples of finite dimensional algebras (of finite and of infinite global dimension) with no semi-simple grading. 

There is some evidence the answer might be yes. We do know the conclusion of the theorem is true, without restriction, for homotopy $K$-theory.

\begin{prop}
Let $\Gamma$ be any basic finite dimensional algebra. Then we have
\[
	\sfK\sfH(\sfD^{\mathrm{b}}(\modu \Gamma)) \cong \sfK\sfH(k)^{\oplus n} \cong \sfK\sfH(\Perf(\Gamma))
\]
where $n$ is the number of simple modules.
\end{prop}
\begin{proof}
Homotopy $K$-theory is nilinvariant for discrete rings by \cite{weibel2013k}*{Corollary~IV.12.5} and so
\[
\sfK\sfH(\Perf \Gamma) \cong \sfK\sfH(\Gamma) \cong \sfK\sfH(\Gamma/\rad(\Gamma)) \cong \sfK\sfH(k)^{\oplus n}.
\]
On the other hand, for $\sfD^{\mathrm{b}}(\modu \Gamma)$ (really its ind-completion) we are in the situation of \cite{antieau2022}*{Corollary~A.2} and so we get an isomorphism $\sfK(\sfD^{\mathrm{b}}(\modu \Gamma)) \stackrel{\sim}{\to} \sfK\sfH(\sfD^{\mathrm{b}}(\modu \Gamma))$. The theorem of the heart and d\'evissage then tell us that
\[
\sfK\sfH(\sfD^{\mathrm{b}}(\modu \Gamma)) \cong \sfK(\sfD^{\mathrm{b}}(\modu \Gamma)) \cong \sfK(k)^{\oplus n} \cong \sfK\sfH(k)^{\oplus n}.
\]
\end{proof}

We also know the statement for periodic cyclic homology $\mathsf{HP}$ for commutative algebras over $\CC$.

\begin{prop}
Let $R$ be a finite dimensional commutative local $\CC$-algebra. Then
\[
\mathsf{HP}(\sfD^{\mathrm{b}}(\modu R)) \cong \mathsf{HP}(\CC) \cong \mathsf{HP}(\Perf R).
\]
\end{prop}
\begin{proof}
The functor $\mathsf{HP}$ is nilinvariant for discrete rings by \cite{Goodwillie} and so
\[
\mathsf{HP}(\Perf R) \cong \mathsf{HP}(R) \cong \mathsf{HP}(\CC).
\]
On the other hand, by \cite{khan2023lattice}*{Theorem~A.2} we have
\[
\mathsf{HP}(\sfD^{\mathrm{b}}(\modu R)) \cong \mathsf{HP}(\sfD^{\mathrm{b}}(\modu \CC)) \cong \mathsf{HP}(\CC).
\]
\end{proof}

Let us make one further remark on this theme. Recall that a localizing invariant $\EE$ is \emph{truncating} if for every connective dg algebra $R$ the canonical map $\EE(R) \to \EE(\sfH^0(R))$ is an isomorphism. By \cite{LandTamme}*{Corollary~3.5} any truncating invariant is nilinvariant for discrete rings and so if $\Gamma$ is a finite dimensional algebra then
\[
\EE(\Perf \Gamma) \cong \EE(\Gamma) \cong \EE(\Gamma/\rad(\Gamma)) \cong \EE(k)^{\oplus n}.
\]
Thus for such invariants everything boils down to computing $\EE(\sfD^{\mathrm{b}}(\modu \Gamma))$. It might be worthwhile to highlight that both $\mathsf{HP}$ in characteristic $0$ and $\sfK\sfH$ are truncating invariants.



%
%
%



%



\bibliography{motive-bib}

\end{document}